\documentclass[11pt,reqno]{amsart}
\usepackage{enumerate, latexsym, amsmath, amsfonts, amssymb, amsthm, color}
\def\pmod #1{\ ({\rm{mod}}\ #1)}
\def\Z{\Bbb Z}

\def\l{\left}
\def\r{\right}
\def\bg{\bigg}
\def\({\bg(}
\def\){\bg)}
\def\t{\text}
\def\f{\frac}

\def\mo{{\rm{mod}\ }}

\def\ls{\leqslant}
\def\gs{\geqslant}

\def\sm{\setminus}
\def\bi{\binom}
\def\al{\alpha}

\def\eq{\equiv}

\def\da{\delta}

\def\Proof{\noindent{\it Proof}}

\theoremstyle{plain}
\newtheorem{theorem}{Theorem}

\newtheorem{lemma}{Lemma}
\newtheorem{corollary}{Corollary}

\theoremstyle{definition}

\theoremstyle{remark}
\newtheorem{remark}{Remark}

\allowdisplaybreaks

 \vspace{4mm}

\begin{document}

\hbox{Finite Fields Appl. 64 (2020), Article 101672.}
\medskip

\title
[{On some determinants involving Jacobi symbols}]
{On some determinants involving Jacobi symbols}

\author
[D. Krachun, F. Petrov, Z.-W. Sun, M. Vsemirnov] {Dmitry Krachun, Fedor Petrov,  Zhi-Wei Sun, Maxim Vsemirnov}

\address {(Dmitry Krachun) St. Petersburg Department of Steklov Mathematical Institute of Russian Academy of Sciences, Fontanka 27, 191023, St. Petersburg, Russia}
\email{dmitrykrachun@gmail.com}

\address {(Fedor Petrov) St. Petersburg Department of Steklov Mathematical Institute of Russian Academy of Sciences, Fontanka 27, 191023, St. Petersburg, Russia}
\email{fedyapetrov@gmail.com}

\address{(Zhi-Wei Sun) Department of Mathematics, Nanjing
University, Nanjing 210093, People's Republic of China}
\email{zwsun@nju.edu.cn}

\address {(Maxim Vsemirnov) St. Petersburg Department of Steklov Mathematical Institute of Russian Academy of Sciences, Fontanka 27, 191023, St. Petersburg, Russia}
\email{vsemir@pdmi.ras.ru}

\keywords{Determinants, Jacobi symbols, character sums over finite fields.
\newline \indent 2020 {\it Mathematics Subject Classification}. Primary 11C20, 11T24; Secondary 11E16, 15A15.
\newline \indent The work is supported by the NSFC (Natural Science Foundation of China)-RFBR (Russian Foundation for Basic Research) Cooperation and Exchange Program (grants NSFC 11811530072 and
RFBR 18-51-53020-GFEN-a). The third author is also supported
by the Natural Science Foundation of China (grant 11971222).}

\begin{abstract}
In this paper we study some conjectures on determinants with Jacobi symbol entries posed by Z.-W. Sun. For any positive integer $n\equiv3\pmod4$, we show that
$$(6,1)_n=[6,1]_n=(3,2)_n=[3,2]_n=0$$
and
$$(4,2)_n=(8,8)_n=(3,3)_n=(21,112)_n=0$$
as conjectured by Sun, where
$$(c,d)_n=\bigg|\left(\frac{i^2+cij+dj^2}n\right)\bigg|_{1\ls i,j\ls n-1}$$
and $$[c,d]_n=\bigg|\left(\frac{i^2+cij+dj^2}n\right)\bigg|_{0\ls i,j\ls n-1}$$
with $(\frac{\cdot}n)$ the Jacobi symbol.
We also prove that $(10,9)_p=0$ for any prime $p\equiv5\pmod{12}$, and $[5,5]_p=0$ for any prime $p\equiv 13,17\pmod{20}$, which were also conjectured by Sun. Our proofs involve character sums over finite fields.
\end{abstract}
\maketitle

\section{Introduction}
\setcounter{lemma}{0}
\setcounter{theorem}{0}
\setcounter{corollary}{0}
\setcounter{remark}{0}
\setcounter{equation}{0}

For an $n\times n$ matrix $[a_{ij}]_{1\ls i,j\ls n}$ over a field, we simply denote its determinant by $|a_{ij}|_{1\ls i,j\ls n}$. In this paper we study some conjectures on determinants with Jacobi symbol entries posed by Z.-W. Sun \cite{S19}.

Let $p$ be an odd prime. In 2004, R. Chapman \cite{Ch} determined the values of
$$\bg|\l(\f{i+j-1}p\r)\bg|_{1\ls i,j\ls(p-1)/2}=\l(\f{-1}p\r)\bg|\l(\f{i+j}p\r)\bg|_{1\ls i,j\ls(p-1)/2}$$
and
$$\bg|\l(\f{i+j-1}p\r)\bg|_{1\ls i,j\ls(p+1)/2}
=\bg|\l(\f{i+j}p\r)\bg|_{0\ls i,j\ls(p-1)/2},$$
where $(\f {\cdot}p)$ denotes the Legendre symbol.
Chapman's conjecture on the evaluation of
$$\bg|\l(\f{j-i}p\r)\bg|_{0\ls i,j\ls(p-1)/2}$$
was confirmed by M. Vsemirnov \cite{V12, V13} via matrix decomposition.
With this background, Z.-W. Sun \cite{S19} studied some new kinds of determinants with Legendre symbol or Jacobi symbol entries.

For any odd integer $n>1$ and integers $c$ and $d$, Sun \cite{S19} introduced the notations
\begin{equation}\label{1.1}(c,d)_n:=\bigg|\l(\f{i^2+cij+dj^2}n\r)\bigg|_{1\ls i,j\ls n-1}\end{equation}
and
\begin{equation}\label{1.2}[c,d]_n:=\bigg|\l(\f{i^2+cij+dj^2}n\r)\bigg|_{0\ls i,j\ls n-1},\end{equation}
where $(\frac{\cdot}n)$ denotes the Jacobi symbol. He showed that
\begin{equation}\label{cd1}\l(\f dn\r)=-1\Rightarrow (c,d)_n=0,
\end{equation}
and that for any odd prime $p$ we have
\begin{equation}\label{cd2}\l(\f dp\r)=1\Rightarrow [c,d]_p=\begin{cases}\f{p-1}2(c,d)_p&\t{if}\ p\nmid c^2-4d,\\\f{1-p}{p-2}(c,d)_p&\t{if}\ p\mid c^2-4d.
\end{cases}
\end{equation}

For $a\in\Z$ and $n\in\Z^+=\{1,2,3,\ldots\}$ , if $a$ is relatively prime to $n$ and $x^2\eq a\pmod n$
for some $x\in\Z$, then $a$ is called a {\it quadratic residue} modulo $n$. If $n$ is odd and $a$ is a quadratic residue modulo $n$, then $(\f an)=1$ since $a$ is a quadratic residue modulo any prime divisor of $n$.

Now we state our first theorem.

\begin{theorem}\label{Th1.1} Let $n>1$ be an odd integer.

{\rm (i)} If $-1$ is not a quadratic residue modulo $n$, then
$$(6,1)_n=(3,2)_n=0\ \ \t{and}\ \ [6,1]_n=[3,2]_n=0.$$

{\rm (ii)} If $-2$ is not a quadratic residue modulo $n$, then
$$(4,2)_n=(8,8)_n=0\ \ \t{and}\ \ [4,2]_n=[8,8]_n=0.$$

{\rm (iii)} If $-3$ is not a quadratic residue modulo $n$, then
$$(3,3)_n=(6,-3)_n=0\ \ \t{and}\ \ [3,3]_n=[6,-3]_n=0.$$

{\rm (iv)} If $-7$ is not a quadratic residue modulo $n$, then
$$(21,112)_n=(42,-7)_n=0\ \ \t{and}\ \ [21,112]_n=[42,-7]_n=0.$$
\end{theorem}

Combining Theorem 1.1 with \eqref{cd1}, we immediately obtain the following consequence
which was conjectured by Sun \cite[Conjecture 4.8(ii)]{S19}.
\begin{corollary}\label{Rem1.1} For any positive integer $n\eq3\pmod4$, we have
$$(6,1)_n=[6,1]_n=(3,2)_n=[3,2]_n=0$$
and
$$(4,2)_n=(8,8)_n=(3,3)_n=(21,112)_n=0.$$
\end{corollary}

Actually we deduce Theorem 1.1 from the following theorems.

\begin{theorem}\label{Th1.2} Let $n$ be a positive odd integer which is squarefree. For any $c,d,i\in\Z$, we have
\begin{equation}\label{1.3}\sum_{j=0}^{n-1}\l(\f jn\r)\l(\f{i^2+cij+dj^2}n\r)
=\sum_{j=0}^{n-1}\l(\f{-j}n\r)\l(\f{i^2+2cij+(c^2-4d)j^2}n\r).
\end{equation}
\end{theorem}

\begin{theorem}\label{Th1.3} Let $n$ be a positive odd integer which is squarefree, and let $i\in\Z$.
Then
\begin{align}\label{1.4}
\sum_{j=0}^{n-1}\l(\f jn\r)\l(\f{i^2+3ij+2j^2}n\r)=0&\quad\t{if}\ -1\ R\ n\ \t{fails},
\\\label{1.5}\sum_{j=0}^{n-1}\l(\f jn\r)\l(\f{i^2+4ij+2j^2}n\r)=0&\quad\t{if}\ -2\ R\ n\ \t{fails},
\\\label{1.6}\sum_{j=0}^{n-1}\l(\f jn\r)\l(\f{i^2+3ij+3j^2}n\r)=0&\quad\t{if}\ -3\ R\ n\ \t{fails},
\\\label{1.7}\sum_{j=0}^{n-1}\l(\f jn\r)\l(\f{i^2+21ij+112j^2}n\r)=0&\quad\t{if}\ -7\ R\ n\ \t{fails},
\end{align}
where the notation $m\ R\ n$ means that $m$ is a quadratic residue modulo $n$.
\end{theorem}

Our following result was originally conjectured by Sun \cite[Conjecture 4.8(iv)]{S19}.

\begin{theorem}\label{Th1.4} {\rm (i)} $(10,9)_p=0$ for any prime $p\eq5\pmod{12}$.

{\rm (ii)} $[5,5]_p=0$ for any prime $p\eq13,17\pmod{20}$.
\end{theorem}

In fact, our proof of Theorem 1.4 yields a stronger result:
For each integer $y$, we have
$$\sum_{x=0}^{p-1}\l(\f{x^5+10x^3y+9xy^2}p\r)=0$$
for any prime $p\eq5\pmod{12}$, and
$$\sum_{x=0}^{p-1}\l(\f{x^5+5x^3y+5xy^2}p\r)=0$$
for any prime $p\eq13,17\pmod{20}$.

We will prove Theorem 1.2, Theorems 1.3 and 1.1, and Theorem 1.4
in Sections 2-4 respectively.

Sun \cite[Conjecture 4.8(iv)]{S19} also conjectured that $(8,18)_p=[8,18]_p=0$
for any prime $p\eq13,17\pmod{24}$. Moreover, Sun \cite{S18} conjectured that
$$\sum_{x=0}^{p-1}\l(\f{x^5+8x^3y+18xy^2}p\r)=0$$
for any prime $p\eq13,17\pmod{24}$ and integer $y$,
and this was confirmed by M. Stoll via two elliptic curves
with complex multiplication by $\Z[\sqrt{-6}]$ (see the answer in \cite{S18}).

For any prime $p\eq1\pmod4$ and $a,b,c\in\Z$, we provide in Section 5 a sufficient condition
for
$$\sum_{x=0}^{p-1}\l(\f{ax^5+bx^3+cx}p\r)=2\sum_{x=1}^{(p-1)/2}\l(\f xp\r)
\l(\f{a(x^2)^2+bx^2+c}p\r)=0.$$

\section{Proof of Theorem 1.2}
\setcounter{lemma}{0}
\setcounter{theorem}{0}
\setcounter{corollary}{0}
\setcounter{remark}{0}
\setcounter{equation}{0}

\begin{lemma}\label{Lem2.1} Let $p$ be an odd prime and let $c,d,i\in\Z$ with $p\nmid c$. Then
\begin{equation}\label{2.1}\sum_{j=0}^{p-1}\l(\f jp\r)\l(\f{i^2+cij+dj^2}p\r)\eq-\l(\f{ci}p\r)\sum_{k=0}^{p-1}\bi{4k}{2k}\bi{2k}k\l(\f{d}{16c^2}\r)^k\pmod p.
\end{equation}
\end{lemma}
\Proof. If $p\mid i$, then both sides of the congruence \eqref{2.1} are zero.

Below we assume $p\nmid i$ and let $L$ denote the left-hand side of the congruence (\ref{2.1}). As $\{ir:\ r=0,\ldots,p-1\}$ is a complete system of residues modulo $p$, we have
\begin{align*}
L=&\sum_{r=0}^{p-1}\l(\f{ir}p\r)\l(\f{i^2+ci(ir)+d(ir)^2}p\r)
=\l(\f{i^3}p\r)\sum_{r=0}^{p-1}\l(\f rp\r)\l(\f{1+cr+dr^2}p\r)
\\\eq&\l(\f ip\r)\sum_{r=1}^{p-1}r^{(p-1)/2}(1+cr+dr^2)^{(p-1)/2}
\\\eq&\l(\f ip\r)\sum_{r=1}^{p-1}\l(r^{-1}+c+dr\r)^{(p-1)/2}\pmod{p}.
\end{align*}
We may write $(x^{-1}+c+dx)^{(p-1)/2}=\sum_{s=-(p-1)/2}^{(p-1)/2}a_sx^s$ with $a_s\in\Z$.
For any integer $s$, it is well known (cf. \cite[p.~235]{IR}) that
\begin{equation}
  \label{eq:sumofpowers}
\sum_{r=1}^{p-1}r^s\eq\begin{cases}-1\pmod p&\t{if}\ p-1\mid s,\\0\pmod{p}&\t{otherwise}.
\end{cases}
\end{equation}
Therefore,
$$\sum_{r=1}^{p-1}\l(r^{-1}+c+dr\r)^{(p-1)/2}
=\sum_{s=-(p-1)/2}^{(p-1)/2}a_s\sum_{r=1}^{p-1}r^s\eq-a_0\pmod p.$$
Clearly,
\begin{align*}a_0=&\sum_{k=0}^{(p-1)/2}\bi{(p-1)/2}{2k}\bi{2k}k c^{(p-1)/2-2k}d^k
\\\eq&\sum_{k=0}^{(p-1)/2}\bi{-1/2}{2k}\bi{2k}k\l(\f cp\r)\l(\f d{c^2}\r)^k
=\l(\f cp\r)\sum_{k=0}^{(p-1)/2}\f{\bi{4k}{2k}\bi{2k}k}{(-4)^{2k}}\l(\f d{c^2}\r)^k
\\=&\l(\f cp\r)\sum_{k=0}^{p-1}\bi{4k}{2k}\bi{2k}k\l(\f{d}{16c^2}\r)^k\pmod p.
\end{align*}
So, by the above, we finally obtain \eqref{2.1}. \qed

\begin{lemma}\label{Lem2.2} Let $p$ be any odd prime. Then we have the congruence
\begin{equation}\label{2.2}\sum_{k=0}^{p-1}\f{\bi{4k}{2k}\bi{2k}k}{64^k}\l(x^k-\l(\f{-2}p\r)(1-x)^k\r)
\eq0\pmod{p^{2-\da_{p,3}}}\end{equation}
in the ring $\Z_p[x]$, where $\Z_p$ is the ring of all $p$-adic integers, and $\da_{p,3}$ is $1$ or $0$ according as $p=3$ or not.
\end{lemma}
\begin{remark}\label{Rem2.1} For any prime $p>3$, the congruence \eqref{2.2} is due to Sun \cite[(1.15)]{S13}. We can easily verify that \eqref{2.2} also holds for $p=3$.
\end{remark}

\medskip
\noindent{\it Proof of Theorem 1.2}. Clearly both sides of \eqref{1.3} vanish if $n=1$.
Below we assume $n>1$ and distinguish three cases.
\medskip

{\it Case} 1. $n$ is an odd prime $p$.
\medskip

Define
$$D:=\sum_{j=0}^{p-1}\l(\f jp\r)\l(\f{i^2+cij+dj^2}p\r)-\sum_{j=0}^{p-1}\l(\f {-j}p\r)\l(\f{i^2+2cij+(c^2-4d)j^2}p\r).$$
If $p\mid c$ and $p\eq3\pmod4$, then
\begin{align*}D=&\sum_{j=1}^{(p-1)/2}\l(\l(\f jp\r)+\l(\f{p-j}p\r)\r)\l(\f{i^2+dj^2}p\r)
\\&-\sum_{j=1}^{(p-1)/2}\l(\l(\f {-j}p\r)+\l(\f{-(p-j)}p\r)\r)\l(\f{i^2-4dj^2}p\r)
\\=&0-0=0.
\end{align*}
When $p\mid c$ and $p\eq1\pmod4$, for $q=((p-1)/2)!$ we have $q^2\eq-1\pmod p$ and $(\f{2q}p)=1$
(cf. \cite[Remark 1.1 and Lemma 2.3]{S19}), thus
\begin{align*}D=&\sum_{j=1}^{p-1}\l(\f jp\r)\l(\f{i^2+dj^2}p\r)-\sum_{j=1}^{p-1}\l(\f {-j}p\r)\l(\f{i^2-4dj^2}p\r)
\\=&\sum_{j=1}^{p-1}\l(\f {2qj}p\r)\l(\f{i^2+d(2qj)^2}p\r)-\sum_{j=1}^{p-1}\l(\f {j}p\r)\l(\f{i^2-4dj^2}p\r)
=0.
\end{align*}

Now suppose that $p\nmid c$. By Lemma 2.2,
\begin{align*}&\sum_{k=0}^{p-1}\bi{4k}{2k}\bi{2k}k\l(\f{d}{16c^2}\r)^k
\\=&\sum_{k=0}^{p-1}\f{\bi{4k}{2k}\bi{2k}k}{64^k}\l(\f{4d}{c^2}\r)^k
\\\eq&\l(\f{-2}p\r)\sum_{k=0}^{p-1}\f{\bi{4k}{2k}\bi{2k}k}{64^k}\l(1-\f{4d}{c^2}\r)^k
\\=&\l(\f{-2}p\r)\sum_{k=0}^{p-1}\bi{4k}{2k}\bi{2k}k\l(\f{c^2-4d}{16(2c)^2}\r)^k\pmod{p^{2-\da_{p,3}}}.
\end{align*}
Combining this with Lemma 2.1, we obtain that
\begin{align*}&\sum_{j=0}^{p-1}\l(\f jp\r)\l(\f{i^2+cij+dj^2}p\r)
\\\eq&-\l(\f{-2ci}p\r)\sum_{k=0}^{p-1}\bi{4k}{2k}\bi{2k}k
\l(\f{c^2-4d}{16(2c)^2}\r)^k
\\\eq&\l(\f{-1}p\r)\sum_{j=0}^{p-1}\l(\f jp\r)\l(\f{i^2+2cij+(c^2-4d)j^2}p\r)\pmod p.
\end{align*}
Thus $D\eq0\pmod p$. Clearly $|D|<2p$.

If $p\mid i$, then
\begin{align*}D=&\sum_{j=0}^{p-1}\l(\f jp\r)\l(\f{dj^2}p\r)-\sum_{j=0}^{p-1}\l(\f{-j}p\r)\l(\f{(c^2-4d)j^2}p\r)
\\=&\l(\l(\f dp\r)-\l(\f{4d-c^2}p\r)\r)\sum_{j=1}^{p-1}\l(\f jp\r)=0.
\end{align*}

Now assume that $p\nmid i$. If neither $c^2-4d$ nor $(2c)^2-4(c^2-4d)=16d$ is divisible by $p$, then
$$|\{1\ls j\ls p-1:\ i^2+cij+dj^2\eq0\pmod p\}|\in\{0,2\}$$
and
$$|\{1\ls j\ls p-1:\ i^2+2cij+(c^2-4d)j^2\eq0\pmod p\}|\in\{0,2\},$$
hence $D$ is even. When $p\mid d$, we also have $2\mid D$ since
$$|\{1\ls j\ls p-1:\ p\mid i(i+cj)\}|=|\{1\ls j\ls p-1:\ p\mid i(i+cj)^2\}|=1.$$
If $p\mid c^2-4d$, then $2\mid D$ since
$$|\{1\ls j\ls p-1:\ p\mid i(i+2cj)\}|=1$$
and
$$\l|\l\{1\ls j\ls p-1:\ i^2+cij+dj^2\eq\l(i+\f c2j\r)^2\eq0\pmod p\r\}\r|=1.$$
So $D$ is always even, and hence $D=0$ as $p\mid D$ and $|D|<2p$.
\bigskip

{\it Case} 2. $n=p_1\ldots p_r$ with $r\gs2$, where $p_1,\ldots,p_r$ are distinct primes.
\smallskip

By the Chinese Remainder Theorem,
\begin{align*}\sum_{j=0}^{n-1}\l(\f jn\r)\l(\f{i^2+cij+dj^2}n\r)
=&\sum_{j=0}^{n-1}\prod_{s=1}^r\l(\f j{p_s}\r)\l(\f{i^2+cij+dj^2}{p_s}\r)
\\=&\sum_{j_1=0}^{p_1-1}\ldots\sum_{j_r=0}^{p_r-1}\prod_{s=1}^r\l(\f {j_s}{p_s}\r)\l(\f{i^2+cij_s+dj_s^2}{p_s}\r)
\end{align*}
and hence
\begin{equation}\label{2.3}\sum_{j=0}^{n-1}\l(\f jn\r)\l(\f{i^2+cij+dj^2}n\r)
=\prod_{s=1}^r\sum_{j_s=0}^{p_s-1}\l(\f{j_s}{p_s}\r)\l(\f{i^2+cij_s+dj_s^2}{p_s}\r).
\end{equation}
Similarly,
\begin{align*}&\sum_{j=0}^{n-1}\l(\f {-j}n\r)\l(\f{i^2+2cij+(c^2-4d)j^2}n\r)
\\=&\prod_{s=1}^r\sum_{j_s=0}^{p_s-1}\l(\f{-j_s}{p_s}\r)\l(\f{i^2+2cij_s+(c^2-4d)j_s^2}{p_s}\r).
\end{align*}
Thus, \eqref{1.3} holds in view of Case 1. This concludes the proof. \qed

\section{Proofs of Theorems 1.3 and 1.1}
\setcounter{lemma}{0}
\setcounter{theorem}{0}
\setcounter{corollary}{0}
\setcounter{remark}{0}
\setcounter{equation}{0}

\begin{lemma}\label{Lem3.1} Let $p>3$ be a prime.
If $p\eq1,3\pmod 8$ and $p=x^2+2y^2$ with $x,y\in\Z$ and $x\eq1\pmod4$, then
$$ \sum_{k=0}^{p-1}\f{\bi{4k}{2k}\bi{2k}k}{128^k}\eq(-1)^{\lfloor(p+5)/8\rfloor}\l(2x-\f p{2x}\r)\pmod {p^2}.$$
If $(\f{-2}p)=-1$, i.e., $p\eq5,7\pmod 8$, then
$$\sum_{k=0}^{p-1}\f{\bi{4k}{2k}\bi{2k}k}{128^k}\eq0\pmod {p^2}.$$
\end{lemma}
\begin{remark} The first assertion in Lemma \ref{Lem3.1} was conjectured by Z.-W. Sun \cite{S11} and confirmed by his twin brother Z.-H. Sun \cite[Theorem 4.3]{Su}. The second assertion was proved by Z.-W. Sun \cite[Corollary 1.3]{S13} as a consequence
of \eqref{2.2} with $x=1/2$.
\end{remark}

\begin{lemma}\label{Lem3.2} Let $p$ be an odd prime and let $c,d,i\in\Z$ with $p\nmid d$. Then
\begin{equation}\label{3.1}\sum_{j=0}^{p-1}\l(\f jp\r)\l(\f{i^2+3cij+dj^2}p\r)=\l(\f ip\r)\sum_{x=0}^{p-1}\l(\f{x^3-(3c^2-d)x+c(2c^2-d)}p\r).
\end{equation}
\end{lemma}
\Proof. Both sides of \eqref{3.1} vanish if $p\mid i$. Below we assume $p\nmid i$.

Clearly,
\begin{align*}&\sum_{j=0}^{p-1}\l(\f jp\r)\l(\f{i^2+3cij+dj^2}p\r)
\\=&\sum_{j=0}^{p-1}\l(\f {dj}p\r)\l(\f{di^2+3ci(dj)+(dj)^2}p\r)
\\=&\sum_{k=0}^{p-1}\l(\f kp\r)\l(\f{k^2+3cik+di^2}p\r)=\sum_{r=0}^{p-1}\l(\f {ir}p\r)\l(\f{(ir)^2+3ci^2r+di^2}p\r)
\\=&\l(\f ip\r)\sum_{r=0}^{p-1}\l(\f rp\r)\l(\f{r^2+3cr+d}p\r)
\end{align*}
and
\begin{align*}\sum_{r=0}^{p-1}\l(\f{r^3+3cr^2+dr}p\r)=&\sum_{x=0}^{p-1}\l(\f{(x-c)^3+3c(x-c)^2+d(x-c)}p\r)
\\=&\sum_{x=0}^{p-1}\l(\f{x^3+(d-3c^2)x+c(2c^2-d)}p\r).
\end{align*}
So \eqref{3.1} holds. \qed

\begin{lemma}\label{Lem3.3} Let $p$ be any odd prime and let $i\in\Z$.

{\rm (i)} We have
\begin{equation}\label{3.2}\begin{aligned}&\sum_{j=0}^{p-1}\l(\f jp\r)\l(\f{i^2+4ij+2j^2}p\r)
\\=&\begin{cases}(-1)^{\lfloor(p-3)/8\rfloor}(\f ip)2x&\t{if}\ p=x^2+2y^2\ (x,y\in\Z\ \&\ 4\mid x-1),\\0&\t{if}\ (\f{-2}p)=-1,\ \t{i.e.},\ p\eq5,7\pmod8.\end{cases}
\end{aligned}
\end{equation}

{\rm (ii)} We have
\begin{equation}\label{3.3}\begin{aligned}&\sum_{j=0}^{p-1}\l(\f jp\r)\l(\f{i^2+3ij+2j^2}p\r)
\\=&\begin{cases}-(\f{2i}p)2x&\t{if}\ p=x^2+4y^2\ (x,y\in\Z\ \&\ 4\mid x-1),\\0&\t{if}\ (\f{-1}p)=-1,\ \t{i.e.},\ p\eq3\pmod4.\end{cases}
\end{aligned}
\end{equation}
Also,
\begin{equation}\label{3.4}\begin{aligned}&\sum_{j=0}^{p-1}\l(\f jp\r)\l(\f{i^2+3ij+3j^2}p\r)
\\=&\begin{cases}-(\f{-i}p)2x&\t{if}\ p=x^2+3y^2\ (x,y\in\Z\ \&\ 3\mid x-1),\\0&\t{if}\ (\f{-3}p)\not=1,\ \t{i.e.},\ p\eq0,2\pmod3,\end{cases}
\end{aligned}
\end{equation}
and
\begin{equation}\label{3.5}\begin{aligned}&\sum_{j=0}^{p-1}\l(\f jp\r)\l(\f{i^2+21ij+112j^2}p\r)
\\=&\begin{cases}-(\f{i}p)2x&\t{if}\ p=x^2+7y^2\ (x,y\in\Z\ \&\ (\f x7)=1),\\0&\t{if}\ (\f{-7}p)\not=1,\ \t{i.e.},\ p\eq0,3,5,6\pmod7.\end{cases}
\end{aligned}
\end{equation}
\end{lemma}
\begin{remark}\label{Rem3.2} It is well known that any prime $p\eq1\pmod4$ can be written as $x^2+4y^2$ with $x,y\in\Z$.
Also, for each $m\in\{2,3,7\}$ any odd prime $p$ with $(\f{-m}p)=1$ can be written  $x^2+my^2$ with $x,y\in\Z$
(cf. \cite{Co}).
\end{remark}

\medskip
\noindent{\it Proof of Lemma 3.3}. It is easy to verify that \eqref{3.2}-\eqref{3.5} hold for $p=3$.
Below we assume $p>3$.

(i) As $16\times4^2/2=128$, combining Lemma 2.1 and Lemma 3.1 we find that
$$\begin{aligned}&\sum_{j=0}^{p-1}\l(\f jp\r)\l(\f{i^2+4ij+2j^2}p\r)
\\\eq&\begin{cases}(-1)^{\lfloor(p-3)/8\rfloor}(\f ip)2x\pmod p&\t{if}\ p=x^2+2y^2\ (x,y\in\Z\ \&\ 4\mid x-1),\\0\pmod p&\t{if}\ (\f{-2}p)=-1,\ \t{i.e.},\ p\eq5,7\pmod8.\end{cases}
\end{aligned}$$
Observe that
$$\sum_{j=0}^{p-1}\l(\f jp\r)\l(\f{i^2+4ij+2j^2}p\r)=\sum_{j=1}^{p-1}\l(\f jp\r)\l(\f{i^2+4ij+2j^2}p\r)$$
is even (since $|\{1\ls j\ls p-1:\ i^2+4ij+2j^2\eq0\pmod p\}|\in\{0,2\}$), and its absolute value is smaller than $p$. If $p=x^2+2y^2$ with $x,y\in\Z$ and $x\eq1\pmod4$,
then $|2x|<2\sqrt p<p$. So \eqref{3.2} holds.

(ii) In light of Lemma 3.2,
\begin{align*}&\sum_{j=0}^{p-1}\l(\f jp\r)\l(\f{i^2+3ij+2j^2}p\r)
\\=&\l(\f ip\r)\sum_{r=0}^{p-1}\l(\f{r^3-r}p\r)
=\l(\f{2i}p\r)\sum_{r=0}^{p-1}\l(\f{2r}p\r)\l(\f{4r^2-4}p\r)
\\=&\l(\f{2i}p\r)\sum_{s=0}^{p-1}\l(\f{s^3-4s}p\r)
\end{align*}
and
\begin{align*}&\sum_{j=0}^{p-1}\l(\f jp\r)\l(\f{i^2+3ij+3j^2}p\r)
\\=&\l(\f ip\r)\sum_{r=0}^{p-1}\l(\f{r^3-1}p\r)
=\l(\f{2i}p\r)\sum_{r=0}^{p-1}\l(\f{(2r)^3-8}p\r)
\\=&\l(\f{2i}p\r)\sum_{s=0}^{p-1}\l(\f{s^3-8}p\r).
\end{align*}
On the other hand, by \cite[Theorem 6.2.9]{BEW} and \cite[pp.\,195-196]{BEW},
$$\sum_{s=0}^{p-1}\l(\f{s^3-4s}p\r)=\begin{cases}-2x&\t{if}\ p=x^2+4y^2\ (x,y\in\Z\ \&\ 4\mid x-1),
\\0&\t{if}\ p\eq3\pmod4,\end{cases}$$
and
$$\sum_{s=0}^{p-1}\l(\f{s^3-8}p\r)=\begin{cases}-2x(\f{-2}p)&\t{if}\ p=x^2+3y^2\ (x,y\in\Z\ \&\ 3\mid x-1),
\\0&\t{if}\ p\eq2\pmod3.\end{cases}$$
So we have \eqref{3.3} and \eqref{3.4}.

Now we prove \eqref{3.5}. Clearly, \eqref{3.5} is valid if $p\mid i$ or $p=7$. Below we assume that $p\nmid i$ and $p\not=7$. Observe that
\begin{align*}&\sum_{j=0}^{p-1}\l(\f jp\r)\l(\f{i^2+21ij+112j^2}p\r)
\\=&\sum_{r=0}^{p-1}\l(\f{112ir}p\r)\l(\f{112i^2+21i^2(112r)+(112ir)^2}p\r)
\\=&\l(\f{i}p\r)\sum_{s=0}^{p-1}\l(\f{s^3+21s^2+112s}p\r).
\end{align*}
By a result of Rajwade \cite{R},
$$\sum_{s=0}^{p-1}\l(\f{s^3+21s^2+112s}p\r)
=\begin{cases}-2x&\t{if}\ p=x^2+7y^2\ (x,y\in\Z\ \&\ (\f x7)=1),\\0&\t{if}\ (\f{-7}p)=-1.\end{cases}$$
Therefore \eqref{3.5} holds.

The proof of Lemma 3.3 is now complete. \qed

\medskip
\noindent{\it Proof of Theorem 1.3}. Write $n=p_1\ldots p_r$ with $p_1,\ldots,p_r$ distinct primes.
In light of \eqref{2.3} and Lemma 3.3(i), if $-2\ R\ n$ fails (i.e., $(\f{-2}{p_s})=-1$ for some $s=1,\ldots,r$) then
$$\sum_{j=0}^{n-1}\l(\f jn\r)\l(\f{i^2+4ij+2j^2}n\r)=\prod_{s=1}^r\sum_{j=0}^{p_s-1}\l(\f{j_s}{p_s}\r)\l(\f{i^2+4ij_s+2j_s^2}{p_s}\r)=0,$$
Thus \eqref{1.5} holds.
Note that if $-2\ R\ n$ then for each $s=1,\ldots,r$ we may write $p_s=x_s^2+2y_s^2$ with $x_s,y_s\in\Z$ and $x_s\eq1\pmod 4$ and hence
$$\sum_{j=0}^{n-1}\l(\f jn\r)\l(\f{i^2+4ij+2j^2}n\r)=\prod_{s=1}^r\l(-1)^{\lfloor(p_s-3)/8\rfloor}\l(\f {i}{p_s}\r)2x_s\r).$$

Similarly, \eqref{1.4}, \eqref{1.6} and \eqref{1.7} also hold in view of \eqref{2.3}
and Lemma 3.3(ii). This concludes our proof of Theorem \ref{Th1.3}. \qed

\medskip
\noindent{\it Proof of Theorem 1.1}. Suppose that $n=\prod_{s=1}^rp_s^{a_s}$, where $p_1,\ldots,p_r$ are distinct primes and $a_1,\ldots,a_r$ are positive integers. If $a_t>1$ with $1\ls t\ls r$, then $n/p_t\eq0\pmod{p_1\ldots p_r}$ and hence for any $i\in\Z$ we have
\begin{align*}\l(\f{i^2+cij+dj^2}{n}\r)=&\prod_{s=1}^r\l(\f{i^2+cij+dj^2}{p_s}\r)^{a_s}
\\=&\prod_{s=1}^r\l(\f{(i+n/p_t)^2+c(i+n/p_t)j+dj^2}{p_s}\r)^{a_s}
\\=&\l(\f{(i+n/p_t)^2+c(i+n/p_t)j+dj^2}{n}\r)
\end{align*}
for all $j=0,\ldots,n-1$. Therefore
$$(c,d)_n=[c,d]_n=0.$$

Below we assume that $n$ is squarefree. If $-1\ R\ n$ fails, then by Theorems \ref{Th1.2} and \ref{Th1.3}
we have
$$\sum_{j=1}^{n-1}\l(\f jn\r)\l(\f{i^2+3ij+2j^2}n\r)=0=\sum_{j=1}^{n-1}\l(\f jn\r)\l(\f{i^2+6ij+j^2}n\r)$$
for all $i=0,\ldots,n-1$, hence $(3,2)_n=(6,1)_n=0$ and $[3,2]_n=[6,1]_n=0$. This proves part (i) of Theorem 1.1.
Similarly, parts (ii)-(iv) of Theorem 1.1 follow from Theorems \ref{Th1.2} and \ref{Th1.3}. This ends the proof. \qed

\section{Proof of Theorem 1.4}
\setcounter{lemma}{0}
\setcounter{theorem}{0}
\setcounter{corollary}{0}
\setcounter{remark}{0}
\setcounter{equation}{0}

Let $q>1$ be a prime power and let $\mathbb F_q$ be the finite field of order $q$. 
A multiplicative character $\chi$ on $\mathbb F_q$ is called {\it trivial} (or {\it principal})
if $\chi(a)=1$ for all $a\in\mathbb F_q^*=\mathbb F_q\sm\{0\}$. For a polynomial
$P(x)=\sum_{s=0}^nc_sx^s\in\mathbb F_q[x]$, we define the homogenous polynomial
\begin{equation}\label{4.1}P^*(x,y)=\sum_{s=0}^nc_sx^{n-s}y^{s}=x^nP\l(\f yx\r).\end{equation}
Fix a list of the elements of $\mathbb F_q$. For a multiplicative character $\chi$ on $\mathbb F_q$, we
introduce the matrices
\begin{equation}\label{4.2}
M(P,\chi)=[\chi(P^*(a,b))]_{a,b\in\mathbb F_q^*}\  \t{and} \ M_0(P,\chi)=[\chi(P^*(a,b))]_{a,b\in\mathbb F_q}.
\end{equation}

\begin{lemma}\label{lm:det1} Let $q>1$ be a prime power and let $\chi$ be a
nontrivial multiplicative
character on $\mathbb{F}_q$.
Suppose that
$P(x)\in \mathbb{F}_q[x]$ and $\sum_{x\in \mathbb{F}_q} \chi(xP(x))=0$.  Then $M(P,\chi)$ is singular
$($i.e., $\det M(P,\chi)=0)$.
If the character
$\chi^{n+1}$ is nontrivial with $n=\deg P$, then the matrix $M_0(P,\chi)$
is singular too.
\end{lemma}

\begin{proof}
We introduce the column vector $v$ whose coordinates are $v_b=\chi(b)$ for
$b\in \mathbb{F}_q^*$. Let
$M=M(P,\chi)$. Then, for any $a\in\mathbb F_q^*$ we have
$$
(Mv)_a=\sum_{b\in\mathbb F_q^*} \chi\left(a^nP\left(a^{-1}b\right)\right)\chi(b)=
\chi(a^{n+1})
\sum_{b\in\mathbb F_q^*} \chi\left(a^{-1}bP\left(a^{-1}b\right)\right)=0.
$$
Since $v$ is a nonzero vector, the matrix $M$ is singular.

Now suppose that the degree of $P$ is $n$ and the character $\chi^{n+1}$ is nontrivial. Let $M_0=M_0(P,\chi)$
and introduce the vector $v$ with coordinates $v_b=\chi(b)$ for
$b\in \mathbb{F}_q$. Then $(M_0v)_a=0$ for all $a\in\mathbb F_q^*$ as before.
Let $c_n$ be the leading coefficient of the polynomial $P(x)$. Then
$$(M_0v)_0=\sum_{b\in\mathbb F_q} \chi(c_nb^n)\chi(b)=\chi(c_n)\sum_{b\in\mathbb F_q} \chi^{n+1}(b)=0.$$
Therefore $M_0v$ is the zero vector and hence $M_0$ is singular.
\end{proof}

Motivated by Lemma 4.1, we give the following more sophisticated lemma.

\begin{lemma}\label{lm:det2} Let $q>1$ be an odd prime power. Suppose that
$g\in \mathbb{F}_q$ is not a square and $\chi$ is a
nontrivial multiplicative
character on $\mathbb{F}_q$ with $\chi(-1)=1$.
Assume that $P(x)\in\mathbb{F}_q[x]$ and
\begin{equation}\label{4.3}\sum_{x\in\mathbb F_q} \chi(xP(x^2))=\sum_{x\in\mathbb F_q} \chi(xP(gx^2))=0.
\end{equation}

{\rm (i)} We have $\operatorname{dim}(\operatorname{Ker}(M(P,\chi)))\gs 2$, in particular $M(P,\chi)$ is singular.

{\rm (ii)} Assume that the character $\chi^{2n+1}$ with $n=\deg P$ is nontrivial. Then
$\operatorname{dim}(\operatorname{Ker}(M_0(P,\chi)))\gs 2$.
\end{lemma}

\begin{proof} For $a,b\in\mathbb{F}_q$, set
$$v_{a,b}:=\begin{cases}\chi(c)=\chi(\sqrt{ab})\ &\t{if}\ ab=c^2\ \t{for some}\ c\in\mathbb F_q,
\\0&\t{otherwise}.\end{cases}$$
This is well defined since $\chi(\pm1)=1$,
The matrix $V=[v_{a,b}]_{a,b\in\mathbb F_q^*}$ has
rank 2; in fact, if $b'=bc^2$ for some $c\in\mathbb F_q$
then columns $b$ and $b'$ in $V$ are proportional, but
columns 1 and $g$ are not proportional.

(i) Write $M$ for $M(P,\chi)$. It suffices to show that $MV$ is the zero matrix.
For $a,b\in\mathbb F_q^*$, the $(a,b)$-entry of the matric $MV$
is
\begin{align*}\sum_{c\in \mathbb F_q}\chi(P^*(a,c))v_{c,b}
=&\sum_{c\in \mathbb F_q\atop bc\ \t{is a square}}\chi\l(a^nP\l(a^{-1}c\r)\r)\chi(\sqrt{bc})
\\=&\f12\sum_{d\in \mathbb F_q}\chi\l(a^nP\l(a^{-1}bd^2\r)\r)\chi(bd)
\\=&\f12\chi(a^nb)\sum_{d\in \mathbb F_q}\chi\l(P_{a^{-1}b}(d)\r),
\end{align*}
where $P_c(x)=xP(cx^2)$ for any $c\in\mathbb F_q$.

Now it remains to show for any $c\in\mathbb F_q^*$ the identity
$$\sum_{x\in \mathbb F_q} \chi(P_c(x))=0.$$
 Clearly, $c=c_0d^2$ for some
$c_0\in\{1,g\}$ and $d\in\mathbb F_q^*$. Thus
\begin{align*}\sum_{x\in \mathbb F_q} \chi(P_c(x))=&\sum_{x\in\mathbb F_q}\chi(xP(c_0d^2x^2))
=\sum_{y\in\mathbb F_q}\chi(d^{-1}yP(c_0y^2))
\\=&\chi(d)^{-1}\sum_{y\in\mathbb F_q}\chi(P_{c_0}(y))=0.
\end{align*}
This proves part (i) of Lemma 4.2.

(ii)  Write $M_0$ for $M_0(P,\chi)$, and define $V_0=[v_{a,b}]_{a,b\in\mathbb F_q}$. (Note the slight difference
between $V_0$ and $V$.) The rank of $V_0$ is still equal to $2$, so it suffices to show that $M_0V_0$ is the zero matrix. Note that the $(a,b)$-entry of $M_0V_0$ is trivially zero if $b=0$ since $v_{c,0}=0$ for all $c\in\mathbb F_q$. For $a,b\not=0$ we can repeat the computation for $MV$ verbatim.
Let $c_n$ denote the leading coefficient of $P(x)$.
If $a=0$ and $b\not=0$, then the $(a,b)$-entry of $M_0V_0$ is
\begin{align*}
\sum_{c\in\mathbb{F}_q}\chi(P^*(0,c))v_{c,b}=&\sum_{\substack{c\in\mathbb{F}_q\\ bc \text{ is a square}}}\chi(c_nc^n)\chi(\sqrt{bc})
\\=&\frac{1}{2}\sum_{d\in\mathbb{F}^*_q}\chi(c_n(b^{-1}d^2)^nd)
=\frac{\chi(b^{-1}c_n)}{2}\sum_{d\in\mathbb{F}_q}\chi^{2n+1}(d).
\end{align*}
This is zero since $\chi^{2n+1}$ is nontrivial. We are done.
\end{proof}

\begin{theorem}\label{differenceofpowers}
Let $q>1$ be an odd prime power and let $m\in\Z^+$ with $\gcd(m,q-1)=1$.
Let $\chi$ be a nontrivial quadratic
character on $\mathbb{F}_q$, and let
\begin{equation}\label{4.4}P_m(x,a)=\sum_{k=0}^{m-1}\bi {2m}{2k+1}a^kx^{m-1-k}
\end{equation}
with $a\in\mathbb F_q^*=\mathbb F_q\sm\{0\}$. Then
\begin{equation}\label{4.5}\sum_{x\in \mathbb{F}_q} \chi(xP_m(gx^2,a))=0\quad\t{for all}\ g\in\mathbb F_q^*.
\end{equation}
If $\chi(-1)=1$, then both $M(P_m(x,a),\chi)$ and $M_0(P_m(x,a),\chi)$ are singular, and moreover either of them
has a kernel of dimension at least two.
\end{theorem}
\begin{proof} In view of Lemma 4.2, we only need to prove \eqref{4.5}.
As $P_m(gx^2,a)=g^{m-1}P_m(x^2,ag^{-1})$ for all $g\in\mathbb F_q^*$, it suffices to show that
\begin{equation}\label{4.6}
\sum_{x\in \mathbb{F}_q} \chi(xP_m(x^2,a))=0
\end{equation} for any $a\in\mathbb F_q^*$.

Clearly, $m$ is odd since $\gcd(m,q-1)=1$.
Recall that $\chi^2$ is the trivial character, and note that
\begin{align*}\sum_{x\in\mathbb F_q}\chi(xP_m(x^2,a))=&\sum_{x\in\mathbb F_q^*}\chi(ax^{-1}P_m((ax^{-1})^2,a))
\\=&\sum_{x\in\mathbb F_q^*}\chi\bg(\sum_{k=0}^{m-1}\bi{2m}{2(m-1-k)+1}a^{2m-1-k}x^{2k+1-2m}\bg)
\\=&\sum_{x\in\mathbb F_q^*}\chi\bg(\sum_{j=0}^{m-1}\bi{2m}{2j+1}a^{m+j}x^{-1-2j}\bg)
\\=&\sum_{x\in\mathbb F_q^*}\chi(a^mx^{-2m}xP_m(x^2,a))=\chi(a)^m\sum_{x\in\mathbb F_q}\chi(xP_m(x^2,a)).
\end{align*}
If $a$ is not a square in $\mathbb F_q$, then $\chi(a)^m=(-1)^m=-1$ and hence (\ref{4.6}) holds by the above.

Now assume that $a=b^2$ with $b\in\mathbb F_q^*$. Since
\begin{align*}\sum_{x\in\mathbb F_q}\chi(xP_m(x^2,a))=&\sum_{x\in\mathbb F_q}\chi(b^{2m-2}xP_m((b^{-1}x)^2,1))
=\chi(b)^{2m-1}\sum_{y\in\mathbb F_q}\chi(yP_m(y^2,1)),
\end{align*}
it remains to show that $\sum_{x\in\mathbb F_q}\chi(xP_m(x^2,1))=0$.
Since $\chi=\chi^{-1}$ and
$$2xP_m(x^2,1)=(x+1)^{2m}-(x-1)^{2m}=((x+1)^m+(x-1)^m)((x+1)^m-(x-1)^m),$$
we have
\begin{align*}&\sum_{x\in\mathbb F_q}\chi(2xP_m(x^2,1))
\\=&\chi(2^{2m})+\sum_{x\in\mathbb F_q\sm\{1\}}\chi((x+1)^m+(x-1)^m)\chi^{-1}((x+1)^m-(x-1)^m)
\\=&1+\sum_{x\in\mathbb F_q\sm\{1\}}\chi\l(\f{(x+1)^m+(x-1)^m}{(x+1)^m-(x-1)^m}\r)
\\=&1+\sum_{x\in\mathbb F_q\sm\{1\}}\chi\l(\f{(1+2/(x-1))^m+1}{(1+2/(x-1))^m-1}\r)
\\=&1+\sum_{y\in\mathbb F_q\sm\{1\}}\chi\l(\f{y^m+1}{y^m-1}\r)
=1+\sum_{y\in\mathbb F_q\sm\{1\}}\chi\l(\f{y+1}{y-1}\r)
\\=&1+\sum_{y\in\mathbb F_q\sm\{1\}}\chi\l(1+\f2{y-1}\r)=1+\sum_{z\in\mathbb F_q\sm\{1\}}\chi(z)=0
\end{align*}
Thus $\sum_{x\in\mathbb F_q}\chi(xP_m(x^2,1))=0$ as desired.

The proof of Theorem 4.1 is now complete. \end{proof}

\medskip
\noindent{\it Proof of Theorem 1.4(i)}. Let $p$ be any prime with $p\eq5\pmod{12}$, and let $\chi$
be the quadratic character of $\mathbb F_p=\Z/p\Z$ with $\chi(x+p\Z)=(\f xp)$ for all $x\in\Z$.
Note that $\chi(-1)=1$ since $p\eq1\pmod4$. Clearly,
$$P_3(x,3)=\bi 61x^2+\bi 633x+\bi 653^2=6(x^2+10x+9).$$
Applying Theorem 4.1, we obtain that
$$(10,9)_p=\det\l[\l(\f{i^2+10ij+9j^2}p\r)\r]_{1\ls i,j\ls p-1}=0$$
and
$$[10,9]_p=\det\l[\l(\f{i^2+10ij+9j^2}p\r)\r]_{0\ls i,j\ls p-1}=0.$$
Note that Sun stated in [S19, Remark 4.9] that $(10,9)_p=0$ if and only if $[10,9]_p=0$.
\qed

\medskip

Let $\mathbb{F}_q$ be a finite field of order $q$. A polynomial $P(x)\in \mathbb{F}_q[x]$
is called a \emph{permutation polynomial} if $P$ is bijective
as a function on $\mathbb{F}_q$. If $\chi$ is a nontrivial multiplicative character on $\mathbb F_q$
and $P(x)\in\mathbb F_q[x]$ is a permutation polynomial, then
$$\sum_{x\in \mathbb{F}_q} \chi(P(x))=\sum_{y\in \mathbb F_q}\chi(y)=0,$$
and also
$$\sum_{x\in \mathbb{F}_q^*} \chi(P(x))=0$$
provided that $P(0)=0$.

\begin{theorem}\label{prop:Dickson}
 Let $q>1$ be an odd prime power and let $m\in\Z^+$ with $\gcd(m,q^2-1)=1$.
Let  $\chi$ be a nontrivial
multiplicative character on $\mathbb F_q$ with $\chi(-1)=1$.
For the polynomial
 \begin{equation}\label{4.7}Q_m(x,a):=\sum_{i=0}^{(m-1)/2}\frac{m}{m-i}\binom{m-i}{i}(-a)^ix^{(m-1)/2-i}
 \end{equation}
 with $a\in\mathbb F_q^*$, we have
  $$\operatorname{dim}(\operatorname{Ker}(M(Q_m(x,a),\chi)))\gs 2.$$
  Moreover, if the character $\chi^{m}$ is nontrivial,
 then
  $$\operatorname{dim}(\operatorname{Ker}(M_0(Q_m(x,a),\chi)))\gs 2.$$
 \end{theorem}
\begin{proof} Let $a\in\mathbb{F}_q$.
 It is a classical result (cf. \cite[pp.\,355-357]{LN}) that
 the Dickson polynomial $D_m(x,a):=xQ_m(x^2,a)$ is a permutation polynomial on $\mathbb{F}_q$.
 For any $g\in\mathbb F_q^*$, as $Q_m(gx^2,a)=g^{(m-1)/2}Q_m(x^2,ag^{-1})$, the polynomial
 $xQ_m(gx^2,a)$ is also a permutation polynomial on $\mathbb F_q$. Thus
 \begin{equation}\label{4.8}\sum_{x\in\mathbb F_q}\chi(xQ_m(gx^2,a))=0\quad\t{for all}\ g\in \mathbb F_q^*.
 \end{equation}
 Combining this with Lemma 4.2, we immediately obtain the desired results.
\end{proof}

\medskip
\noindent{\it Proof of Theorem 1.4(ii)}.  Let $p$ be any prime with $p\eq13,17\pmod{20}$.
Then $\gcd(5,p^2-1)=1$. Let $\chi$
be the quadratic character of $\mathbb F_p=\Z/p\Z$ with $\chi(x+p\Z)=(\f xp)$ for all $x\in\Z$.
Then $\chi(-1)=1$ since $p\eq1\pmod4$. Clearly  $\chi^5=\chi$ is nontrivial and $Q_5(x,-1)=x^2+5x+5$.
Applying Theorem 4.2, we get that
$$[5,5]_p=\det\l[\l(\f{i^2+5ij+5j^2}p\r)\r]_{0\ls i,j\ls p-1}=0.$$
This concludes the proof. \qed

Note that actually our method to prove Theorem 1.4 yields a stronger result stated after Theorem 1.4 in Section 1.

\section{A sufficient condition for $\sum_{x=0}^{p-1}(\f{ax^5+bx^3+cx}p)=0$}
\setcounter{lemma}{0}
\setcounter{theorem}{0}
\setcounter{corollary}{0}
\setcounter{remark}{0}
\setcounter{equation}{0}

For an odd prime power $q>1$, we let $\chi_q$ denote the quadratic multiplicative character on the finite field $\mathbb{F}_q$.

Let $p\eq1\pmod4$ be a prime and let $a$ be a nonzero element of $\mathbb F_p$. If $\chi_p(a)=1$,
then we define $\sqrt a$ as an element $\al\in\mathbb F_p$ with $\al^2=a$.
When $\chi_p(a)=-1$, the finite field $\mathbb F_{p^2}\cong \mathbb F_p[x]/(x^2-a)$ contains an element
$\al$ with $\al^2=a$, and we denote such an $\al\in\mathbb F_{p^2}$ by $\sqrt{a}$.

\begin{theorem}
  \label{thm:elliptic}
Let $p\equiv 1\pmod{4}$ be a prime and let $a,b,c$ be nonnzero elements of the field $\mathbb{F}_p$. Let $q$ be $p$ or $p^2$ according as $\chi_p(ac)$ is $1$ or $-1$, and set
$$\gamma =\frac{b+2\sqrt{ac}}{16\sqrt{ac}}\in\mathbb F_q.$$
 Let $N$ be the number of $\mathbb{F}_q$-points on the affine curve
$y^2=x^4+x^2+\gamma$. If $N\equiv -1 \pmod{p}$, then
$$
\sum_{x\in \mathbb{F}_p} \chi_p (ax^5+bx^3+cx) = 0.
$$
\end{theorem}

For the sake of convenience, for an odd prime $p$ we
introduce the following two polynomials over $\mathbb{F}_p$:
\begin{equation}
 \label{eq:deff}
 f(z) =  1 + \sum_{k=1}^{ \lfloor (p-1)/8 \rfloor} \prod_{j=0}^{k-1}
 \frac{(8j+1)(8j+5)}{4(j+1)(4j+3)} z^k,
\end{equation}
and
 \begin{equation}
 \label{eq:defg}
  g(z)=1 + \sum_{k=1}^{ \lfloor (p-1)/4 \rfloor}
 \frac{(4k-1)!!}{4^k (k!)^2 } z^k.
\end{equation}

\begin{lemma}
 \label{lem:Ap}
Let $p$ be a prime with $p\eq1\pmod 4$, and let $a,b,c\in \mathbb{F}_p\setminus\{0\}$.
Define
$$A_p=\sum_{x\in\mathbb F_p}\chi_p(ax^5+bx^3+cx).$$
Viewing $A_p \pmod{p}$ as
an element of $\mathbb{F}_p$, we have
$$
  A_p \pmod{p} =
   -\binom{(p-1)/2}{(p-1)/4}b^{(p-1)/4} (a^{(p-1)/4}+c^{(p-1)/4})
    f\l(\f{ac}{b^2}\r).
$$
Consequently, $A_p=0$ if $(a^{-1}c)^{(p-1)/4}=-1$ or $f(ac/b^2)=0$.
\end{lemma}

\begin{proof} As $A_p=\sum_{x\in \mathbb F_p\sm\{0\}}\chi_p(ax^5+bx^3+cx)$, we have $|A_p|<p$. So,
the second assertion in Lemma \ref{lem:Ap} follows from the first one.

Now we come to prove the first assertion. With the help of \eqref{eq:sumofpowers}, in $\mathbb F_p$ we have
\begin{align*}
A_p\pmod p =& \sum_{x\in \mathbb{F}_p} (ax^5+bx^3+cx)^{(p-1)/2} \\
=& \sum_{k_5+k_3+k_1=(p-1)/2} \frac{((p-1)/2)!}{k_5!k_3!k_1!} a^{k_5} b^{k_3} c^{k_1} \sum_{x=0}^{p-1} x^{5k_5+3k_3+k_1} \\
=&  - \sum_{\genfrac{}{}{0pt}{2}{k_5+k_3+k_1=(p-1)/2}{5k_5+3k_3+k_1=p-1}}
    \frac{((p-1)/2)!}{k_5!k_3!k_1!} a^{k_5} b^{k_3} c^{k_1}
\\&    - \sum_{\genfrac{}{}{0pt}{2}{k_5+k_3+k_1=(p-1)/2}{5k_5+3k_3+k_1=2(p-1)}}
    \frac{((p-1)/2)!}{k_5!k_3!k_1!} a^{k_5} b^{k_3} c^{k_1}.
\end{align*}
(Note that if $k_1,k_3,k_5$ are nonnegative integers with $k_1+k_3+k_5=(p-1)/2$ then
$k_1+3k_3+5k_5\ls 5(k_1+k_3+k_5)<3(p-1).$) Thus,
\begin{align*}
&A_p\pmod p
\\=& - \sum_{\genfrac{}{}{0pt}{2}{k_5+k_3+k_1=(p-1)/2}{k_3+2k_5=(p-1)/4}}
    \frac{((p-1)/2)!}{k_5!k_3!k_1!} a^{k_5} b^{k_3} c^{k_1}
    - \sum_{\genfrac{}{}{0pt}{2}{k_5+k_3+k_1=(p-1)/2}{k_3+2k_1=(p-1)/4}}
    \frac{((p-1)/2)!}{k_5!k_3!k_1!} a^{k_5} b^{k_3} c^{k_1}
    \\
 =&  - \sum_{k=0}^{\lfloor(p-1)/8\rfloor}
    \frac{((p-1)/2)!}{k!((p-1)/4-2k)!((p-1)/4+k)!} a^k b^{(p-1)/4-2k} c^{(p-1)/4+k} \\
&- \sum_{k=0}^{\lfloor (p-1)/8\rfloor}
    \frac{((p-1)/2)!}{((p-1)/4+k)!((p-1)/4-2k)!k!} a^{(p-1)/4+k} b^{(p-1)/4-2k} c^{k}
  \\
=& -\binom{(p-1)/2}{(p-1)/4}b^{(p-1)/4} (c^{(p-1)/4}+a^{(p-1)/4}) \\
&\times
  \( 1 + \sum_{k=1}^{\lfloor (p-1)/8\rfloor}
    \prod_{i=0}^{2k-1} \l(\f{p-1}4-i\r)\cdot  \prod_{j=1}^{k}
    \frac{1}{((p-1)/4+j)} \cdot\frac{1}{k!} \l(\f{ac}{b^2}\r)^k \)
   \\
=&-\binom{(p-1)/2}{(p-1)/4}b^{(p-1)/4} (a^{(p-1)/4}+c^{(p-1)/4})
   f\l(\f{ac}{b^2}\r)
\end{align*}
as desired.

\end{proof}

\begin{lemma}
\label{lem:extension}
Let $p$ be an odd prime and let $q=p^n$ with $n\in\Z^+$.
For any polynomial
$$
 H(x)=\sum_{k=0}^{2(p-1)} c_k x^k \in \mathbb F_q[x],
$$
we have
\begin{equation}
   \sum_{x\in \mathbb{F}_q} H(x)^{1+p+\cdots+p^{n-1}} = - c_{p-1}^{1+p+\cdots+p^{n-1}} - c_{2(p-1)}^{1+p+\cdots+p^{n-1}}.
\label{eq:sumzFq}
\end{equation}
\end{lemma}

\begin{proof} As the multiplicative group $\mathbb F_q\setminus\{0\}$ is cyclic,
similar to (\ref{eq:sumofpowers}), for each $s=0,1,2,\ldots$ we have
$$
   \sum_{x\in \mathbb{F}_q} x^s =
   \begin{cases}
      -1&\mbox{if}\ s\in (q-1)\Z^+, \\
       0&\mbox{otherwise},
    \end{cases}
$$
where we treat $0^0$ as $1$ when $s=0$. Note also that
$$H(x)^{p^i}=
\sum_{k=0}^{2(p-1)} c_k^{p^i} x^{kp^i}
$$
for all integers $i\gs0$. Thus
$$
   \sum_{x\in \mathbb F_q}H(x)^{1+p+\cdots+p^{n-1}}  =
   \sum_{x\in \mathbb{F}_q} \prod_{i=0}^{n-1} \biggl( \sum_{k_i=0}^{2(p-1)} c_{k_i}^{p^i} x^{k_ip^i} \biggr)
    = - {\sum\nolimits_{k_0,\dots,k_{n-1}}^{\ast}} \prod_{i=0}^{n-1} c_{k_i}^{p^i}
$$
where $\sum^\ast$ means that the sum is taken over all $k_0,\dots,k_{n-1}\in\{0,1,\ldots,2p-2\}$ subject to the condition
\begin{equation}
\label{eq:q-1div}
k_0 +k_1p+\cdots+k_{n-1}p^{n-1}\in(q-1)\Z^+.
\end{equation}

Write $k_i=p-1+t_i$, where
$-(p-1)\ls t_i \ls p-1$. Then
$$\sum_{i=0}^{n-1}k_ip^i= q-1+\sum_{i=0}^{n-1}t_ip^i.$$
Note that
$$
   \biggl| \sum_{i=0}^{n-1} t_i p^i \biggr| \ls q-1,
$$
and the equality is possible only if $t_0=\cdots=t_{n-1}=p-1$ (i.e., $k_0=\cdots=k_{n-1}=2(p-1)$)
or $t_0=\cdots=t_{n-1}=-(p-1)$ (i.e., $k_0 +k_1p+\cdots+k_{n-1}p^{n-1} = 0$).
Since $|t_i|<p$, if $\sum_{i=0}^{n-1} t_i p^i =0$ then we obtain step by step that
$t_0=\cdots=t_{n-1}=0$ (i.e., $k_0=\cdots=k_{n-1}=p-1$).

Combining the above, we finally obtain \eqref{eq:sumzFq}.
\end{proof}

\begin{lemma}
\label{lem:Bq}
Let $p$ be an odd prime and let $q=p^n$ with $n\in\Z^+$.
 Let $\alpha,\beta,\gamma\in \mathbb{F}_q\setminus\{0\}$ and set
$$B_q=\sum_{x\in \mathbb{F}_q} \chi_q (\alpha x^4+\beta x^2+\gamma).$$
Viewing $B_q \pmod{p}$ as an element of $\mathbb{F}_q$, we have
\begin{equation}\label{Bqequal}
 B_q \pmod{p} =  - \chi_q(\alpha) - \chi_q(\beta) g\l(\f{\alpha\gamma}{\beta^2}\r)^{1+p+\cdots+p^{n-1}}.
\end{equation}
\end{lemma}
\begin{proof}  Write
$$H(x):=(\alpha x^4+\beta x^2+\gamma )^{(p-1)/2}=\sum_{k=0}^{2(p-1)}c_kx^k.$$
In view of Lemma~\ref{lem:extension}, we have
\begin{equation}\label{Bq1}\begin{aligned}
  B_q\pmod p=& \sum_{x\in \mathbb{F}_q} (\alpha x^4+\beta x^2+\gamma )^{(q-1)/2}=\sum_{x\in \mathbb F_q}H(x)^{1+p+\cdots+p^{n-1}}
  \\=&-c_{p-1}^{1+p+\cdots+p^{n-1}}-c_{2(p-1)}^{1+p+\cdots+p^{n-1}}.
\end{aligned}\end{equation}

Clearly,
\begin{equation}\label{Bq2}c_{2(p-1)}^{1+p+\cdots+p^{n-1}}=\alpha^{(q-1)/2} =\chi_q(\alpha).
\end{equation}
 Note also that
\begin{align*}
 c_{p-1}=& \sum_{\genfrac{}{}{0pt}{2}{k_4+k_2+k_0=(p-1)/2}{4k_4+2k_2=p-1}}
    \frac{((p-1)/2)!}{k_4!k_2!k_0!} \alpha^{k_4} \beta^{k_2} \gamma^{k_0} \\
=& \sum_{0\ls k \ls (p-1)/4} \frac{((p-1)/2)!}{((p-1)/2-2k)!(k!)^2} \alpha^k \beta^{((p-1)/2-2k)} \gamma^{k}  \\
=&\beta^{(p-1)/2}+ \beta^{(p-1)/2} \sum_{k=1}^{\lfloor (p-1)/4\rfloor} \prod_{j=0}^{2k-1} \l(\f{p-1}2-j\r) \cdot
     \frac{1}{(k!)^2} \l(\f{\alpha\gamma}{\beta^2}\r)^{k}  \\
=& \beta^{(p-1)/2} g\l(\f{\alpha\gamma}{\beta^2}\r)
\end{align*}
and hence
\begin{equation}\label{Bq3}
c_{p-1}^{1+p+\cdots+p^{n-1}} =\chi_q(\beta)g\l(\f{\alpha\gamma}{\beta^2}\r)^{1+p+\cdots+p^{n-1}}.
\end{equation}

Combining \eqref{Bq1} with \eqref{Bq2} and \eqref{Bq3}, we immediately obtain the desired \eqref{Bqequal}.
\end{proof}

Now we study further properties of the polynomials $f$ and $g$ defined by (\ref{eq:deff}) and (\ref{eq:defg}). They may be viewed as truncated versions of certain hypergeometric series.

\begin{lemma}
\label{lem:diffeq}
Let  $p$ be an odd prime and let $q=p^n$ with $n\in\Z^+$.

{\rm (i)} A polynomial $u\in \mathbb{F}_q[z]$ with $\deg u \ls \lfloor (p-1)/4 \rfloor$  satisfies
the differential equation
\begin{equation}
\label{eq:diffeqg}
  (4z-16z^2) u'' +(4-32z)u'-3u =0
\end{equation}
if and only if $u=cg$ for some $c\in \mathbb{F}_q$.

{\rm (ii)} Suppose that $p\equiv 1 \pmod{4}$. 
Then a polynomial $v\in \mathbb{F}_q[z]$ with $\deg v \ls \lfloor (p-1)/8 \rfloor$
satisfies the differential equation
\begin{equation}
\label{eq:diffeqf}
 (16 z-64 z^2) v'' +(12-112z) v'-5v =0
\end{equation}
if and only if $v=cf$ for some $c\in \mathbb{F}_q$.
\end{lemma}
\begin{proof}
It is straightforward to verify that $u=g$ and $v=f$ satisfy (\ref{eq:diffeqg}) and (\ref{eq:diffeqf}) respectively. So, the ``if" parts of (i) and (ii) are easy.

Now we prove the ``only if" part of (i). If a polynomial $u\in \mathbb F_q[z]$ with $\deg u\le\lfloor (p-1)/4\rfloor$ satisfies (\ref{eq:diffeqg}), then there is a constant $c\in\mathbb F_q$ such that
$\tilde u=u-cg$ is a solution of (\ref{eq:diffeqg}) with $\deg \tilde u<\lfloor (p-1)/4\rfloor$.
 Thus,
it suffices to show that (\ref{eq:diffeqg}) has no nonzero solution $u=c_d z^d + \cdots + c_0$ with
$\deg u =d <  \lfloor (p-1)/4 \rfloor$. In fact, the coefficient of $z^d$
 in $(4z-16z^2) u'' +(4-32z)u'-3u$ is $-(4d+1)(4d+3)c_d\neq 0$ provided $d <  \lfloor (p-1)/4 \rfloor$.

Similarly, we can show the ``only if" part of (ii).
\end{proof}

\begin{lemma}
 \label{lem:fg} Let $p=4n+1$ be a prime with $n\in\Z^+$. Then
$$
 -(2n)! (n!)^2  g(z)= (16z-2)^{n} f\left(\frac{1}{(16z-2)^2}\right).
$$
\end{lemma}

\begin{proof}
Clearly, $u=(16z-2)^{n} f((16z-2)^{-2})$ is a polynomial of degree $n=(p-1)/4$ with the leading coefficient $1$.
A direct computation based on \eqref{eq:diffeqf} shows that $u$ satisfies \eqref{eq:diffeqg}.
Now we apply Lemma~\ref{lem:diffeq} and compare the leading terms of both sides.
Since
$$
  \frac{(p-2)!!}{4^{n}(n!)^2} = \frac{(p-1)!}{2^{p-1} (2n)! (n!)^2}
  \equiv - \frac{1}{(2n)! (n!)^2} \pmod{p},
$$
we immediately get the desired result.
\end{proof}

\begin{proof}[Proof of Theorem~\ref{thm:elliptic}]
Since
$$
N=\sum_{x\in \mathbb F_q}(1+\chi_q(x^4+x^2+\gamma))=q+\sum_{x\in\mathbb{F}_q} \chi_q(x^4+x^2+\gamma),
$$
the assumption $N\equiv -1\pmod{p}$, together with Lemma~\ref{lem:Bq} in the case $\alpha=\beta=1$, implies that $g(\gamma)=0$. As $(16\gamma-2)^{-2}=ac/b^2$, we have
$f(ac/b^2)=0$ by Lemma~\ref{lem:fg}. Applying Lemma~\ref{lem:Ap} we obtain the desired result.
\end{proof}

\setcounter{conjecture}{0} \end{document}